\documentclass{siamart220329}
\usepackage{amsfonts}
\usepackage{graphicx}
\usepackage{epstopdf}
\usepackage{mathabx}
\usepackage{algorithmic}
\usepackage{algorithm}
\usepackage{enumerate}
\usepackage{bm}
\usepackage{tikz}
\usepackage{diagbox}
\usepackage{comment}
\usetikzlibrary{decorations.pathreplacing,calligraphy}
\usetikzlibrary{matrix}
\usetikzlibrary{shapes,patterns,arrows,snakes,decorations.shapes}
\usetikzlibrary{positioning,fit,calc}
\usetikzlibrary{plotmarks}
\usepackage{multirow}
\newtheorem{example}{Example}[section]

\ifpdf
  \DeclareGraphicsExtensions{.eps,.pdf,.png,.jpg}
\else
  \DeclareGraphicsExtensions{.eps}
\fi

\newsiamremark{remark}{Remark}
\newsiamremark{hypothesis}{Hypothesis}
\crefname{hypothesis}{Hypothesis}{Hypotheses}
\newsiamthm{claim}{Claim}

\headers{Convergence Properties of NGMRES}{C. Greif and Y. He}

\title{Convergence Properties of Nonlinear GMRES Applied to Linear Systems}

\author{Chen Greif\thanks{Department of Computer Science, The University of British Columbia, Vancouver, BC, Canada (\email{greif@cs.ubc.ca}). The work of the first author was supported in part by Discovery Grant RGPIN-2023-05244 of the Natural Sciences and Engineering Research Council of Canada.}
  \and Yunhui He\thanks{Department of Mathematics, University of Houston, 3551 Cullen Blvd, Room 641, Houston, Texas 77204-3008, USA (\email{yhe43@central.uh.edu}).}}

\usepackage{amsopn}

\ifpdf
\hypersetup{
    pdftitle={Convergence Properties of Nonlinear GMRES Applied to Linear Systems},
	pdfauthor={Chen Greif, Yunhui He}
}
\fi

\begin{document}

\maketitle

\begin{abstract}
 The Nonlinear GMRES (NGMRES) proposed by Washio and Oosterlee [Electron. Trans. Numer. Anal, 6(271-290), 1997] is an acceleration method for fixed point iterations. 
 It has been demonstrated to be effective, but its convergence properties have not been extensively studied in the literature so far. 
 In this work we aim to close some of this gap, by offering a convergence analysis for NGMRES applied to linear systems. A central part of our analysis focuses on identifying equivalences between NGMRES and the classical Krylov subspace GMRES method. 
 
\end{abstract}

 
 \section{Introduction} \label{sec:intro}
	
	Consider solving the system of equations
	\begin{equation}
		g(x)=0,
	\end{equation}
	with a fixed-point iteration: given an initial guess, $x_0$, 
	\begin{equation}\label{eq:FP}
		x_{k+1}=q(x_k), \qquad k=0,1,2,\dots
	\end{equation}
	In practice,  \eqref{eq:FP} may slowly converge or even diverge, depending on the derivative of $q(x)$ and the proximity to a fixed point. Methods have been developed to accelerate \eqref{eq:FP}; examples here are nonlinear GMRES (NGMRES) \cite{oosterlee2000krylov,sterck2012nonlinear,washio1997krylov} 
    and Anderson acceleration (AA) \cite{walker2011anderson,anderson1965iterative}, among other instances. 
    
     NGMRES was originally proposed by Washio and Oosterlee ~\cite{washio1997krylov} as an accelerator for multigrid applied to nonlinear partial differential equations, and extended to application to recirculating flows using nonlinear multigrid methods \cite{oosterlee2000krylov} and American-style options \cite{oosterlee2003multigrid}. Later on, it was studied for other nonlinear problems and applications, such as the alternating least-squares method for solving tensor decomposition problems \cite{sterck2012nonlinear,sterck2013steepest,sterck2021asymptotic}. In \cite[Section 2.2]{sterck2012nonlinear}, the author provides an interpretation of NGMRES as the GMRES method \cite{saad1986gmres,loe2022polynomial}. 
     For linear problems,  both methods minimize the two-norm of the residual over the search subspace, and have some additional properties in common, which we refer to throughout the paper. However, the convergence properties of NGMRES have not been broadly explored for linear systems. It has been observed in \cite{sterck2021asymptotic} that the performance of NGMRES is competitive with AA, which has been studied more extensively in the literature \cite{ni2009anderson,tang2024anderson,anderson2019comments,de2024anderson,de2022linear,pollock2019anderson,higham2016anderson,potra2013characterization,yang2022anderson,pratapa2016anderson,suryanarayana2019alternating,chen2022composite,ji2023improved}.

Since an important component in this paper is related to identifying connections between NGMRES and GMRES, we present the enormously popular GMRES method in Algorithm \ref{alg:GMRES}, following \cite{saad2003iterative}, 
for solving a linear system of the form 
	\begin{equation}\label{eq:linear-system}
		Ax=b, \qquad {\rm where } \quad A\in \mathbb{R}^{n\times n},  b \in \mathbb{R}^n. 
	\end{equation}
	\begin{algorithm}[t!]
\caption{GMRES for solving a linear system $Ax=b$ \cite[Algorithm 6.9]{saad2003iterative}} \label{alg:GMRES}
		\begin{algorithmic}[1] 
			\REQUIRE $x_0$  and $k>0$
           \STATE compute $r_0 = b -Ax_0, \beta= \|r_0\|_2$, and $u_1 = r_0/\beta$
			\FOR {$j=1,\cdots$ until $k$ }
				\STATE compute  $w_j = Au_j$
				   \FOR {$i=1,\cdots$ until $j$ }
                      \STATE   $ h_{ij} = (w_j , u_i)$
                      \STATE   $w_j = w_j -h_{ij}u_i$
                      \ENDFOR
			     \STATE $h_{j+1,j}= \|w_j\|_2$. If $h_{j+1,j} = 0$ set $k= j$ and go to 11
                 \STATE $u_{j+1} = w_j/h_{j+1,j}$
			\ENDFOR
             \STATE Define the $(k + 1) \times k$ Hessenberg matrix $\Bar{H}_k = \{h_{ij}\}_{1\leq i \leq k+1,1\leq j\leq k}$  
              \STATE   Compute $y_k$ the minimizer of $\|\beta e_1-\Bar{H}_ky\|_2$ and $x_k = x_0 + V_k y_k$.
		\end{algorithmic}
	\end{algorithm}
 Next, let us provide a brief introduction to NGMRES.

	\begin{algorithm}[h!]
		\caption{Nonlinear GMRES: NGMRES($m$)} 
            \label{alg:NGMRES}
		\begin{algorithmic}[1]
			\REQUIRE $x_0$  and $m\geq0$
			\FOR {$k=0,1,\cdots$ until convergence}
				\STATE compute 
				\begin{equation}\label{eq:xkp1}
					x_{k+1} = q(x_k) + \sum_{i=0}^{m_k}\beta_i^{(k)} \left(q(x_k)-x_{k-i} \right),
				\end{equation}
				where $m_k=\min\{k,m\}$ and $\beta_i^{(k)}$ is obtained by solving the  least-squares problem
				\begin{equation}\label{eq:min}
					\min_{\big(\beta_0^{(k)},\beta_1^{(k)},\cdots, \beta_{m_k}^{(k)}\big)} \left\|g(q(x_k))+\sum_{i=0}^{m_k} \beta_i^{(k)} \left(g(q(x_k))-g(x_{k-i}) \right) \right\|_2^2.
				\end{equation}
			\ENDFOR
		\end{algorithmic}
	\end{algorithm}

Based on \eqref{eq:FP}, we define the residual at the $k$th iterate as
	\begin{equation}\label{eq:defrk}
		r(x_k)=x_k-q(x_k).
	\end{equation}
NGMRES for accelerating \eqref{eq:FP} is given in Algorithm \ref{alg:NGMRES}.  
We denote its windowed version with window size $m$ as NGMRES($m$),  
and its untruncated version, i.e., with $m_k=k$ for each $k$, as full NGMRES.

In this work, we aim to close a gap in the understanding of NGMRES by providing a detailed study of the performance of this method when applied to linear systems. 

Consider a linear system of the form \eqref{eq:linear-system}, and denote its solution by $x^* \in \mathbb{R}^n$.
    We will study the fixed-point iteration  
	\begin{equation}\label{eq:linear-q}
		q(x)=Mx+b,
	\end{equation}
	where 
    \begin{equation}\label{eq:M}
    M=I-A.
    \end{equation}
    Note that $r(x)=x-q(x)=x-(Mx+b)=Ax-b=g(x)$. For the solution $x^*$ we have $r(x^*)=0,$ 
        and given that $r(x)=g(x)=0=Ax-b$, we also have $g(x^*)=0$.

 We note with regard to Algorithm \ref{alg:NGMRES} that since $g(x)=r(x)$, we can rewrite \eqref{eq:min} as
	\begin{equation}
		\min_{\big(\beta_0^{(k)},\beta_1^{(k)},\cdots, \beta_{m_k}^{(k)}\big)} \left\|r(q(x_k))+\sum_{i=0}^{m_k} \beta_i^{(k)} \left(r(q(x_k))- r(x_{k-i}) \right) \right\|_2^2.
        \label{eq:minresidual}
	\end{equation}
The minimization problem \eqref{eq:minresidual} that is explained in  \cite[Section 2.2]{sterck2012nonlinear} is equivalent to the mechanism of GMRES. In both cases, the task at hand is to determine the coefficients ($\{ \beta_i^{(k)} \}$ in the case of \eqref{eq:minresidual}) for which the norm is minimized, and  we are dealing with a changed subspace as we iterate.
 On the other hand, a fundamental difference is that classical GMRES features a highly efficient incremental solution of the corresponding least-squares problem, related to the Arnoldi process. Its nonlinear counterpart does not possess the attractive properties of generating orthogonal bases and minimizing the two-norm of residuals.

We comment that in this work for the linear case, we limit ourselves to considering the fixed-point iteration in the form of  \eqref{eq:linear-q}. However, in practice, one may use other fixed-point iterations. For the nonlinear case, a simple fixed-point iteration is $q(x_k)=x_k-g(x_k)$, which corresponds to the steepest descent if $g(x)=\nabla f(x)$. In general, for the nonlinear case, $q(x)$ might be problem-dependent.

    \begin{remark}
        As evident by step 2 of Algorithm \ref{alg:NGMRES}, our definition of the $k$-step NGMRES  starts from $k=0$ and generates $x_{k+1}$. In contrast,  as can be observed in step 1 of Algorithm \ref{alg:GMRES}, $k$-step classical GMRES starts with $x_0$ and computes $x_k$. So, there is an index shift here, which we will account for in our analysis.
    \end{remark}
     \begin{remark} In this work we denote $r_k=r(x_k)=Ax_k-b$ as the residual for all methods considered, even though standard practice for classical GMRES  is to refer to the negated quantity as the residual; see step 2 of Algorithm~\ref{alg:GMRES}. 
        \end{remark}

 Finally, we introduce AA in Algorithm \ref{alg:AA}. In the algorithm,  the residual $r(x_k)$ is as defined in \eqref{eq:defrk}. We refer to the case where $m_k=k$ for each $k$ as full AA.

	\begin{algorithm}[h!]
		\caption{Anderson Acceleration:  AA($m$)} \label{alg:AA}
		\begin{algorithmic}[1] 
			\REQUIRE $x_0$  and $m\geq0$
			\FOR {$k=0,1,\cdots$ until convergence }
				\STATE compute 
				\begin{equation}\label{eq:xkp1-AA}
					x_{k+1} = q(x_k) + \sum_{i=1}^{m_k}\gamma_i^{(k)} \left(q(x_k)-q(x_{k-i}) \right),
				\end{equation}
				where $m_k=\min\{k,m\}$ and $\gamma_i^{(k)}$ is obtained by solving the  least-squares problem
				\begin{equation}\label{eq:min-AA}
					\min_{\big(\gamma_1^{(k)},\cdots, \gamma_{m_k}^{(k)}\big)} \left\|r(x_k)+\sum_{i=1}^{m_k} \gamma_i^{(k)} \left(r(x_k)-r(x_{k-i}) \right) \right\|_2^2.
				\end{equation}
			\ENDFOR
		\end{algorithmic}
	\end{algorithm}
 
Let us establish some notation that will come handy throughout the paper. We denote $x_j^A, x_j^{NG}$ and $x_j^G$ as the iterates generated by full AA, full NGMRES, and GMRES, respectively, and analogously their residuals: $r_j^A, r_j^{NG}$ and $r_j^G$. 

    It has been shown that the iterates generated from full AA  applied to a linear system using the fixed-point iteration \eqref{eq:linear-q} can be recovered from classical GMRES \cite{saad2003iterative}  (assuming exact arithmetic), i.e., $x_{k+1}^A=q(x_k^G)$,  
    under certain  conditions \cite{walker2011anderson,ni2009anderson,potra2013characterization}.
	As for NGMRES, in \cite{sterck2012nonlinear} the author discussed similarities between the two methods, but the question whether NGMRES applied to linear systems is identical to GMRES was not fully addressed. Indeed, in the literature there seems to be no detailed characterization of the relationship between NGMRES($m$) and classical GMRES even for solving linear systems, neither theoretically nor experimentally. We note, though, that our study is not intended to promote the use of NGMRES($m$) as a comparable linear solver to GMRES, but rather to understand its properties.
	
    The main contributions of this work are as follows. Assuming that the 2-norm of the residual of GMRES is strictly decreasing, then:
	\begin{enumerate}[(i)]
		\item 
        we prove that if $A$ is invertible, the iterates generated by full NGMRES and GMRES are identical;
		\item we derive some orthogonality properties for NGMRES;
        \item for the finite-window version, when $A$ defined in \eqref{eq:linear-system} is either symmetric or shifted skew-symmetric, we prove that NGMRES($1$) is mathematically equivalent to  GMRES.  
		\item we derive upper bounds on the convergence rate of NGMRES for certain cases.
	\end{enumerate}

	The remainder of this paper is organized as follows. In section \ref{sec:comparison}, we prove (i). In section \ref{sec:NGMRESm}, we establish (ii) and (iii). In section \ref{sec:conv}, we provide upper bounds  as per (iv).  In section \ref{sec:num}, we briefly discuss a few implementation details and present some numerical experiments to validate our theoretical findings.  Finally, we  draw some conclusions in section \ref{sec:con}.

	\section{Comparison of AA, GMRES, and NGMRES for solving linear systems}\label{sec:comparison}
 
      For NGMRES, let $$\boldsymbol{\beta}^{(k)}=\big(\beta_0^{(k)},\beta_1^{(k)},\cdots, \beta_{m_k}^{(k)}\big).$$  
    Given the linear system \eqref{eq:linear-system} and fixed-point iteration \eqref{eq:linear-q}, we claim that 
	\begin{equation}\label{GMRES-LSP-rk}
		\min_{\boldsymbol{\beta}^{(k)}}  \left\|g(q(x_k))+\sum_{i=0}^{m_k} \beta_i^{(k)} \left(g(q(x_k))-g(x_{k-i}) \right) \right\|_2^2= \min_{\boldsymbol{\beta}^{(k)}} \left\| r_{k+1}  \right\|_2^2.
	\end{equation}
	This means that NGMRES in each step minimizes the residual.
	By standard calculations, we have
	\begin{align*}
		&  \min_{\boldsymbol{\beta}^{(k)}}  \left\|g(q(x_k))+\sum_{i=0}^{m_k} \beta_i^{(k)} \left(g(q(x_k))-g(x_{k-i}) \right) \right\|_2^2 \\
		= &  \min_{\boldsymbol{\beta}^{(k)}} \left\|A q(x_k)-b+\sum_{i=0}^{m_k} \beta_i^{(k)} \left(A q(x_k)-b-A x_{k-i}+b) \right) \right\|_2^2 \\
		= &  \min_{\boldsymbol{\beta}^{(k)}} \left\|A q(x_k)-b+\sum_{i=0}^{m_k} \beta_i^{(k)} \left(A q(x_k)-A x_{k-i}) \right) \right\|_2^2 \\
		= &  \min_{\boldsymbol{\beta}^{(k)}}  \left\| A \left( q(x_k)+ \sum_{i=0}^{m_k} \beta_i^{(k)} \left( q(x_k)- x_{k-i} \right) \right) -b \right\|_2^2 \\
		= &  \min_{\boldsymbol{\beta}^{(k)}}  \left\| A x_{k+1} -b \right\|_2^2 \\
		= &  \min_{\boldsymbol{\beta}^{(k)}} \left\| r_{k+1}  \right\|_2^2.
	\end{align*}

We start by presenting recurrence relations that the residuals satisfy among themselves.
	\begin{theorem}
		The residuals of NGMRES applied to fixed-point iteration \eqref{eq:linear-q} satisfy the recurrence relation
		\begin{equation} r_{k+1} = \left(1+\sum_{i=0}^{m_k} \beta_i^{(k)} \right) M r_k - \sum_{i=0}^{m_k}  \beta_i^{(k)} r_{k-i}. 
			\label{eq:rkp1}
		\end{equation}
	\end{theorem}
	
	\begin{proof}
		Using $r_k=A x_k-b,$ by \eqref{eq:xkp1} we have
		\begin{align*}
			r_{k+1}  = &A q(x_k) + \sum_{i=0}^{m_k} \beta_i^{(k)} \left(A q(x_k)-A x_{k-i} \right) - b\\
			= & A \left[\left(1+\sum_{i=0}^{m_k} \beta_i^{(k)} \right) M x_k
			+\left(1+\sum_{i=0}^{m_k} \beta_i^{(k)} \right) b  - \sum_{i=0}^{m_k}\beta_i^{(k)}x_{k-i} \right] - b\\
			= & M \left(1+\sum_{i=0}^{m_k} \beta_i^{(k)} \right)   A x_k
			+\left(1+\sum_{i=0}^{m_k} \beta_i^{(k)} \right) Ab - \beta_0^{(k)} A x_k - \beta_1^{(k)} A x_{k-1}\\
			&-\cdots-
			\beta_{m_k}^{(k)} Ax_{k-m_k}  - b\\
			= & M \left(1+\sum_{i=0}^{m_k} \beta_i^{(k)} \right)   (r_k+b)
			+\left(1+\sum_{i=0}^{m_k} \beta_i^{(k)} \right) Ab - \sum_{i=0}^{m_k}\beta_i^{(k)} (r_{k-i}+b)  - b\\
			= &  \left(1+\sum_{i=0}^{m_k} \beta_i^{(k)} \right)  M r_k
			+\left(1+\sum_{i=0}^{m_k} \beta_i^{(k)} \right) b - \beta_0^{(k)} (r_k+b) - \beta_1^{(k)} (r_{k-1}+b)\\
			&-\cdots-
			\beta_{m_k}^{(k)} (r_{k-m_k}+b)  - b\\
			= &  \left(1+\sum_{i=0}^{m_k} \beta_i^{(k)} \right)  M r_k
			-\sum_{i=0}^{m_k} \beta_i^{(k)}   r_{k-i}.
		\end{align*}
	\end{proof}
	
	\begin{definition}
    We denote the Krylov subspace associated with the iterations of the methods we consider by
	\begin{equation}
		\mathcal{K}_s={\rm span}\{r_0, Ar_0, A^2r_0,\cdots, A^{s-1}r_0\}.
	\end{equation}
    \end{definition}

	\begin{lemma}\cite[Lemma 2.4]{walker2011anderson} \label{lem:GMRES-mono}
		Suppose that GMRES is applied to linear system \eqref{eq:linear-system} with a nonsingular $A$. If $\|r_{k-1}^G\|_2 >\|r_k^G\|_2>0$ for some $k>0$, then $r_k^G\notin \mathcal{K}_k$ 
	\end{lemma}

	In \cite{walker2011anderson}, the authors have shown that the iterates obtained from full AA have a direct relationship with the iterates of GMRES. We state the result below.
	\begin{lemma}\cite[Theorem 2.2]{walker2011anderson}\label{lem:AA-GMRES}
		Assume that $A$ is invertible, and GMRES and full AA use the same initial guess $x_0\neq x^*$. Furthermore, assume for some $k_0> 0$,
		$r_{k_0-1}^{G}\neq 0$ and $\|r_{k-1}^G\|_2 >\|r_k^G\|_2$ for each $k$ such that $0 < k < k_0$. Then, $x_{j+1}^A=q(x_j^G)=x_j^G-r_j^G$, where $0\leq j \leq k_0$, and $r_j^G=Ax_j^G-b$. 
	\end{lemma}
\begin{remark} The value of $k_0$ depends on $x_0$ and $A$. Let $\mu(A, x_0)$ be the smallest integer $t$ for which there is a nonzero polynomial $p(z)$ of degree $t$ such that $p(A)r_0 = 0$, and  let $\nu(A, r_0)$ be the maximum integer $s$ such that the set $\{x^A_j-x_0\}_{j=1}^s$ is linearly independent.  In \cite{potra2013characterization}, it is shown that if  $\mu(A, x_0)<\nu(A, r_0)$, then GMRES stagnates (i.e., it produces two identical successive iterates) at the index $\mu(A, x_0)$. If this occurs, GMRES will continue expanding the Krylov subspace and eventually converge to the exact solution, whereas Anderson acceleration will stagnate.
\end{remark}

Next, we derive a relationship among full AA, full NGMRES and GMRES.
	
	\begin{theorem}\label{thm:relation-infinity}
		Assume that full AA, full NGMRES and classical GMRES use the same initial guess $x_0\neq x^*$. Furthermore, assume  for some $k_0> 0$, $r_{k_0-1}^{G}\neq 0$ and $\|r_{k-1}^{G}\|_2 >\|r_k^{G}\|_2$ for each $k$ such that $0 < k< k_0$. Then,  if $A$ is invertible,  $x_j^{NG}=x_j^G$ and   $x_{j+1}^A=q(x_j^{NG})=x_j^{NG}-r_j^{NG}$, where $0\leq j \leq k_0$.
	\end{theorem}
Before giving the proof, we comment that the value of $k_0$ in Theorem ~\ref{thm:relation-infinity} depends on $x_0$ and $A$. If $k_0=1$,  the theorem does not tell us anything about the performance of full NGMRES. If GMRES stagnates, the behavior of NGMRES is  difficult to characterize, and we do not further explore theoretical results in this direction. Instead, in Section \ref{sec:num} we illustrate the difficulty on an example with different initial guesses; see Example \ref{ex:circ-staga}. 
        
	\begin{proof}
		First, for a given $j$, where $1\leq j\leq k_0$, we define 
		\begin{equation}\label{eq:basis1}
			z_i=x_{i}^{NG}-x_0, \quad  i=1,2,\cdots, j-1, 
	     \end{equation}
			and
		\begin{equation}\label{eq:basis2}
		z_{j}=x_{ j-1}^{NG}-x_0-r_{j-1}^{NG}.
		\end{equation}
		To prove that 
        $x_j^{NG}=x_j^G$, it is sufficient to prove the following two Claims:
		\begin{enumerate}
			\item For $1\leq j\leq k_0$, if $\{z_1,\cdots, z_j\}$ is a basis for $\mathcal{K}_j$, then 
             $x_j^{NG}=x_j^G$.
			\item For $1\leq j\leq k_0$,  $\{z_1,\cdots, z_j\}$ is a basis for $\mathcal{K}_j$.
		\end{enumerate}
		{\em Proof of Claim 1.}  Let $r_0=Ax_0-b$. For a given $j$, where $1\leq j\leq k_0$, from \eqref{eq:basis1} and \eqref{eq:basis2}, we have
		\begin{equation*}
		Az_i=r_{i}^{NG}-r_0, \quad  i=1,2,\cdots, j-1, 
		\end{equation*}
		and 
		\begin{equation*}
			Az_{ j}=Mr_{j-1}^{NG}-r_0. 
		\end{equation*}
		Then, using \eqref{eq:rkp1}  and the above two equations, for the $(j-1)$-step full NGMRES, i.e., $m_{j-1}= j-1$, we have
		\begin{align*}
	r_{j}^{NG}&=A x_{j}^{NG} -b\\
			    &= \left(1+\sum_{i=0}^{j-1} \beta_i^{(j-1)} \right)  M r_{j-1}^{NG}
			    -\sum_{i=0}^{j-1} \beta_i^{(j-1)}  r_{j-1-i}^{NG}\\
			    &=\left(1+\sum_{i=0}^{j-1} \beta_i^{(j-1)} \right)(Az_{j}+r_0)
			    -\sum_{i=0}^{j-2} \beta_i^{(j-1)} (Az_{j-1-i}+r_0)-\beta_{j-1}^{(j-1)}r_0\\
			    &=r_0+A \left(1+\sum_{i=0}^{j-1} \beta_i^{(j-1)} \right)z_{j}
			    - A\left(\sum_{i=0}^{j-2} \beta_i^{(j-1)} z_{j-1-i}\right)	\\
			    &=r_0-A\left(\sum_{i=1}^{j} \alpha_i z_i\right), 
		\end{align*}
		where $\alpha_i=\beta_{j-1-i}^{(j-1)}$ for $i=1, 2, \cdots, j-1$,  and $\alpha_{j}=-(1+\sum_{i=0}^{j-1} \beta_i^{(j-1)}) $. We emphasize that there is no constraint on $\alpha_i$ for $i=1, 2, \cdots, j$. Let $\boldsymbol{\alpha}=(\alpha_1,\cdots,\alpha_{j})^T$.
		Using \eqref{eq:xkp1}, \eqref{eq:basis1} and \eqref{eq:basis2}, it can be shown that $x_j^{NG}=x_0-\sum_{i=1}^{j} \alpha_i z_i$. We know that the $(j-1)$-step full NGMRES solves the problem
		\begin{equation}\label{eq:LSPrksame}
			\min_{\boldsymbol{\beta}^{(j-1)}}  \left\| A x_{j}^{NG} -b \right\|_2^2
			=\min_{\boldsymbol{\alpha}}\left\| r_0-A \left(\sum_{i=1}^{j} \alpha_i z_i \right)\right\|_2^2=
			\min_{u\in \mathcal{K}_{j}} \left\| r_0-Au\|_2^2, \right.
		\end{equation}
		which is exactly what $j$-step GMRES does.  By the assumption that $\{z_i\}_{i=1}^j$ is a basis for $\mathcal{K}_j$ and $A$ is invertible, we have $x_{j}^{NG}=x_{j}^G$.
		
		{\em Proof of Claim 2.}
        For $j=1$, since $z_1=x_0^{NG}-x_0-r_0\neq 0$, $z_1$ is a basis for $\mathcal{K}_1$. Suppose that $Z_j=\{x_1^{NG}-x_0, x_2^{NG}-x_0, \cdots, x_{j-1}^{NG}-x_0,x_{j-1}^{NG}-x_0-r_{j-1}^{NG}\}$ is a basis for $\mathcal{K}_j$ for all $j$ such that $1<j< k_0$. Next, we prove that $Z_{j+1}=\{x_1^{NG}-x_0, x_2^{NG}-x_0, \cdots, x_{j-1}^{NG}-x_0,x_{j}^{NG}-x_0,x_{j}^{NG}-x_0-r_{j}^{NG}\}$ is a basis for $\mathcal{K}_{j+1}$. 
		Since we know $Z_j$ is a basis for $\mathcal{K}_{j}$, from Claim 1 we  have $r_{j}^{NG}=r_{j}^G$. Since $x_{j}^{NG}-x_0 \in \mathcal{K}_{j}$ and  $r_{j}^G \in \mathcal{K}_{j+1}$, we have $x_{j}^{NG}-x_0-r_j^G\in \mathcal{K}_{j+1}$. Moreover, since $\|r_{j-1}^G\|_2 >\|r_{j}^G\|_2\neq 0$ by assumption, Lemma \ref{lem:GMRES-mono} states $r_{j}^G\notin \mathcal{K}_{j}$. It follows that $x_{j}^{NG}-x_0-r_j^G\notin \mathcal{K}_j$.  Since $x_1^{NG}-x_0, \, x_2^{NG}-x_0, \cdots, \,x_{j-1}^{NG}-x_0$ are linearly independent and belong to $\mathcal{K}_{j-1}$, as well as  $x_{j}^{NG}-x_0 \in \mathcal{K}_{j}$ and $x_{j}^{NG}-x_0 \notin\mathcal{K}_{j-1}$ (otherwise $A(x_j^{NG}-x_0)=r_j^{NG}-r_0 \in \mathcal{K}_j$, which contradicts $r_j^{NG}=r_j^G\notin\mathcal{K}_{j}$), we conclude that the first $j$ elements in $Z_{j+1}$ form a basis for $\mathcal{K}_j$. Recall $ x_{j}^{NG}-x_0-r_j^{NG}\notin \mathcal{K}_{j}$, but it is in $\mathcal{K}_{j+1}$. As a result, $Z_{j+1}$ is a basis for $\mathcal{K}_{j+1}$. From  Lemma \ref{lem:AA-GMRES}, we have the desired result for NGMRES and AA.  
	\end{proof}

	\begin{corollary}
		Suppose that the assumptions of Theorem \ref{thm:relation-infinity} hold and that $r_{k_0}^G=r_{k_0-1}^G\neq 0$. Then, $r_{k_0}^{NG}=r_{k_0-1}^{NG}$.
	\end{corollary}

    \begin{remark}[Stagnation]
In Theorem \ref{thm:relation-infinity}, in some situations $k_0$ might not exist, i.e., GMRES might stagnate at the first iteration; see Example \ref{ex:circ-staga} for an instance of this.  When that happens, the behavior of GMRES is complicated \cite{nachtigal1992fast}  and it is harder to establish connections with NGMRES. There seems to be no general way to describe the behavior of the residuals $r_j^G$ and $r_j^{NG}$ for $j>k_0>1$.    
\end{remark}

Theorem \ref{thm:relation-infinity}  extends to the preconditioned case, as follows.
	\begin{corollary}
		Let $A = P-N$, where $P$ is nonsingular, and consider
		GMRES applied to the left-preconditioned system  
       $P^{-1}Ax = P^{-1}b.$  Define $r_{j,pr}=P^{-1}Ax_j-P^{-1}b$. Let full NGMRES be applied to accelerate the fixed-point iteration $q_{pr}(x)=(I-P^{-1}A)x+P^{-1}b$ for solving the left-preconditioned system.
		Assume that full AA,  full NGMRES and classical GMRES use the same initial guess $x_0\neq x^*$. Furthermore, assume for some $k_0> 0$,  $r_{k_0-1,pr}^{NG}\neq 0$ and $\|r_{k-1,pr}^{NG}\|_2 >\|r_{k,pr}^{NG}\|_2$ for each $k$ such that $0 < k < k_0$. If $A$ is invertible, we have
        $x_j^{NG}=x_j^G$, and $x_{j+1}^A=q_{pr}(x_j^{NG})=x_j^{NG}-r_{j,pr}^{NG}$, where $0\leq j \leq k_0$.
	\end{corollary}
 
We note that a similar result holds for the right-preconditioned case. Specifically, consider the full NGMRES applied to $q(u)=(I-AP^{-1})u+b$,  i.e., we solve 
$AP^{-1}u=b. $
Define $r_k^{NG}=AP^{-1}u_k^{NG}-b$.  Consider full GMRES applied to 
the right-preconditioned system. Assume that the  2-norm of GMRES residual $r_k^G=AP^{-1}u_k^G-b$ is strictly decreasing and $A$ is invertible. Then $u_j^{NG}=u_j^G$.

	\section{Equivalence of NGMRES($m$) and GMRES}\label{sec:NGMRESm}
	In this section, we explore some properties of windowed NGMRES. We start with a polynomial relation for the residuals.
	
	\begin{theorem}\label{thm:NGMRESm-poly}
		For NGMRES($m$) applied to \eqref{eq:linear-q},  the residuals satisfy 
		$$ r_{k+1} = p_{k+1} (M) r_0, \quad k\ge 0,$$
		where $p_{k+1}(\lambda)$ is a polynomial of degree at most $k + 1$ with $p_{k+1}(0)=0$ and $p_{k+1}(1)=1$, satisfying the following recurrence relation: 
		\begin{align*}
			p_{k+1}(\lambda) & = \left(1+\sum_{i=0}^{m_k} \beta_i^{(k)} \right) \lambda p_k(\lambda) - \sum_{i=0}^{m_k} \beta_i^{(k)} p_{k-i}(\lambda)  \\
			&=\left( \left(1+\sum_{i=0}^{m_k} \beta_i^{(k)} \right) \lambda +\beta_0^{(k)}\right) p_k(\lambda) - \sum_{i=1}^{m_k} \beta_i^{(k)} p_{k-i}(\lambda).
		\end{align*}
	\end{theorem}
	 The result can be derived easily; thus, we omit the proof. 
	We point out that Theorem \ref{thm:NGMRESm-poly} holds for full NGMRES, i.e., $m_k=k$.
	\subsection{Orthogonality properties}
	Next, we study orthogonality properties of NGMRES. We first mention some orthogonality properties of GMRES, which might be known or can be easily inferred but are useful to present here for the purpose of comparison with NGMRES. 

    Recall that GMRES is founded upon requiring
	\begin{equation}\label{eq:GMRES-orth-space}
		r_k^G \perp A\mathcal{K}_k.
	\end{equation}
	 The following orthogonality properties of GMERS are useful for our later proof of Theorem \ref{thm:gmres1=gmres}.  
	\begin{lemma}\label{lem:GMRES-ortho}
		The residuals  of GMRES satisfy 
		\begin{equation}\label{eq:GMRES-pro}
			(r_{k+1}^G)^TAr_k^G =0,\quad   (r_{k+1}^G)^T(r_j^G-r_i^G) =0,  \quad 0 \leq  i\leq j\leq k+1.
		\end{equation}
	\end{lemma}
	\begin{proof}
		We know that $r_k^G\in \mathcal{K}_{k+1}$. From \eqref{eq:GMRES-orth-space}, we have $(r_{k+1}^G)^TAr_k^G =0$. Notice that $r_j^G-r_i^G=A(x_j^G-x_i^G)=A(x_j^G-x_0^G+x_0^G-x_i^G)$, as well as $x_j^G-x_0^G \in \mathcal{K}_j$ and $x_i^G-x_0^G \in \mathcal{K}_i$.  It follows that $x_j^G-x_i^G\in \mathcal{K}_j$ and $r_j^G-r_i^G \in A\mathcal{K}_j$. Since $0\leq i\leq j\leq k+1$ and $\mathcal{K}_j\subset \mathcal{K}_{k+1}$,  the desired result follows.
	\end{proof}
        \begin{remark}
	If we take $j=k+1$ and $i=k$ in \eqref{eq:GMRES-pro},  using the Cauchy-Schwarz inequality, this property indicates that the norm of residuals of GMRES are nonincreasing, as expected from a residual-minimization process.
    \end{remark}
	For notational convenience,  for AA let $\boldsymbol{\gamma}^{(k)}=\big(\gamma_1^{(k)},\cdots, \gamma_{m_k}^{(k)}\big)$. One can show that if $M$ is invertible, the least-squares problem \eqref{eq:min-AA} in AA can be rewritten as
	\begin{equation}\label{eq:AA-min}
		\min_{\boldsymbol{\gamma}^{(k)}} \left\| M^{-1} r^A_{k+1}  \right\|_2^2,
	\end{equation}
	where $r^A_{k+1}=Ax^A_{k+1}-b$. For full AA, using Lemma \ref{lem:AA-GMRES}, we have 
	\begin{equation*}
		r_{k+1}^A=Ax_{k+1}^A-b=A(x_k^G-r_k^G)-b=r_k^G-Ar_k^G=Mr_k^G. 
	\end{equation*}
	
	Thus, we have the following result.
	\begin{theorem}
		Assume that the conditions in Lemma \ref{lem:AA-GMRES} hold and $M$ is invertible. For full AA,  
		\begin{equation}\label{eq:AA-min-G}
			\min_{\boldsymbol{\gamma}^{(k)}} \left\| M^{-1} r^A_{k+1}  \right\|_2^2=\min\|r_k^G\|_2
		\end{equation}
		and 
		\begin{equation}\label{eq:AA-orth-space}
			M^{-1}r_{k+1}^A \perp A\mathcal{K}_k.
		\end{equation}  
		Moreover, if $M$ is symmetric, then 
		\begin{equation}\label{eq:AA-orth-space-sym}
			r_{k+1}^A \perp  M^{-1}A\mathcal{K}_k.
		\end{equation} 
	\end{theorem}
    The proof is straightforward; thus we omit it. 
 
	Let $R_k$ be the coefficient matrix of the least-squares problem \eqref{eq:min-AA} in AA, given by 
	\begin{equation*}
		R_k=\begin{bmatrix} r_k^A-r_{k-1}^A & r_k^A-r_{k-2}^A, & \cdots, & r_k^A-r_{k-m_k}^A \end{bmatrix}.
	\end{equation*}
	In \cite[Proposition 5]{de2024anderson} it is shown that the residuals of AA($m$) satisfy 
	\begin{equation*}
		R_k^TM^{-1}r_{k+1}^A=0,   
	\end{equation*}
	 which leads to
	\begin{equation}\label{eq:M-orth-AA}
		(M^{-1}r_{k+1}^A)^T(r_{k-j}^A-r_{k-i}^A)=0, \quad i,j=0, 1,2,\cdots, m_k.
	\end{equation}
 This indicates that  the $(k+1)$st residual of AA($m$) is $M^{-1}$-orthogonal to the difference of two residuals that belong to the set $\{ r_k^A, \cdots, r_{k-m_k}^A\}$. 	When $m_k=k$ (i.e., full AA), the above result is equivalent to \eqref{eq:AA-orth-space}.  

 Moreover, for the residuals of AA (see \cite{de2024anderson}) we have
\begin{equation}\label{eq:r_k-AA-recur}
    r_{k+1}^A =  Mr_{k}^A + \sum_{i=0}^{m_k} \gamma_i^{(k)}M(r_k^A-  r_{k-i}^A).
\end{equation}
Using \eqref{eq:M-orth-AA} and \eqref{eq:r_k-AA-recur}, we have
\begin{equation}
    (M^{-1}r_{k+1}^A)^T (M^{-1}r_{k+1}^A)=(M^{-1}r_{k+1}^A)^Tr_k^A,
\end{equation}
which means that $\|r_k^A\|\geq\|M^{-1}r_{k+1}^A\|$. For a general $M$, there is no guarantee that the norms of residuals of AA are nonincreasing, which is a significant difference compared to NGMRES and GMRES.

	We now derive an analogous orthogonality property for NGMRES. Recall that $m_k={\rm min}(m,k)$.  We drop the superscript, NG, for NGMRES for simplicity.  The following results hold for both full NGMRES and NGMRES($m$).  Based on our discussion so far, the minimization problem \eqref{eq:min} can be rewritten as
	\begin{equation}\label{eq:LSP-linear}
		\min_{\boldsymbol{\beta}^{(k)}} \left\| r_{k+1}  \right\|_2^2= \min_{\beta_i^{(k)}}  \left\| M r_k+ \sum_{i=0}^{m_k} \beta_i^{(k)} M r_k -\sum_{i=0}^{m_k} \beta_i^{(k)}  r_{k-i}\right\|_2^2.
	\end{equation}
	Let us define
	\begin{equation}\label{eq:definition-Wk}
		W_k=  \begin{bmatrix}
			r_k-Mr_k&  r_{k-1}-Mr_k &\cdots & r_{k-m_k}-Mr_k.
		\end{bmatrix}
	\end{equation}
	The  normal equations for the least-squares problem \eqref{eq:LSP-linear} are given by
	\begin{equation}\label{eq:normal-equ}
		(W^T_k W_k)	\boldsymbol{\beta}^{(k)}= W^T_k Mr_k
	\end{equation}
	and if $W_k$ has full rank, the solution is 
	\begin{equation}\label{eq:beta-form}
		\boldsymbol{\beta}^{(k)}= (W^T_k W_k)^{-1} W^T_k Mr_k.
	\end{equation}

    From this it follows that NGMRES has a similar orthogonality property to that of AA, as follows.
	\begin{theorem}\label{thm:orthogonality}
		The residuals of NGMRES  satisfy  
		$$ r_{k+1}^T W_k=0,$$
		that is, 
		\begin{equation}\label{eq:orthogonality-dif}
			r_{k+1}^TAr_k=0, \quad r_{k+1}^T(r_{k-j}-r_{k-i})=0, \quad i,j=0, 1,2,\cdots, m_k.
		\end{equation}
		Moreover,
		\begin{equation}\label{eq:NGMRESm-decrease}
			r_{k+1}^T(r_{k+1}-r_k)=0,
		\end{equation}
		which means that either $r_{k+1}=r_k$ or $\|r_k\|>\|r_{k+1}\|$. 
	\end{theorem}
	Before we give the proof, we point out that for $m_k=k$, i.e., full GMRES, the results in Theorem \ref{thm:orthogonality} are the same as these in Lemma \ref{lem:GMRES-ortho} for GMRES, which is  consistent with the fact that full NGMRES and GMRES coincide.  The orthogonality property \eqref{eq:NGMRESm-decrease} indicates that for both windowed versions of NGMRES, i.e., NGMRES($m$) and full NGMRES, the norms of the residuals of NGMRES are nonincreasing. For windowed NGMRES, Theorem \ref{thm:orthogonality} indicates that the ($k$+1)st residual of NGMRES($m$) is orthogonal to the difference of two previous residuals that belong to the set $\{r_{k+1}, r_k, \cdots, r_{k-m_k}\}$. We comment that NGMRES can be interpretated as a multisecant method \cite{fang2009two}; like AA does \cite{yang2022anderson}. More details can be found in \cite[Section 2]{he2025convergence}.

	\begin{proof}
		  The minimization problem \eqref{eq:LSP-linear} of NGMRES can be rewritten as 
		\begin{align*}
			\min_{\boldsymbol{\beta}^{(k)}} \left\| r_{k+1}  \right\|_2^2= \min_{\beta_i^{(k)}}  \left\| M r_k-W_k\bm{\beta}^{(k)}\right\|_2^2.
		\end{align*}
		Thus,
		$$ W_k^T (M r_k-W_k\boldsymbol{\beta}^{(k)})=W_k^T r_{k+1}=0,$$
		which gives the desired result. The above equality also indicates that
		\begin{equation}\label{eq:temp-result-orth}
			r_{k+1}^T(r_{k-i}-Mr_k)=0, \quad i=0, 1,\cdots, m_k.
		\end{equation}
		However, using $M=I-A$, it follows that  $r_{k-i} -M r_k=r_{k-i} - r_k+Ar_k$. For $i=0$, \eqref{eq:temp-result-orth} gives 
		\begin{equation*}
			r_{k+1}^TAr_k=0.
		\end{equation*}
		For $ i=1,\cdots, m_k$, \eqref{eq:temp-result-orth} gives $r_{k+1}^T(r_k-r_{k-i})=0$. For $i,j=0,1,\cdots, m_k$, it follows that 
		\begin{equation*}
			r_{k+1}^T(r_{k-j}-r_{k-i})=r_{k+1}^T(r_k-r_{k-i})-r_{k+1}^T(r_k-r_{k-j})=0,
		\end{equation*}
		which gives the desired result.
		
		Next, we prove \eqref{eq:NGMRESm-decrease}. From \eqref{eq:rkp1}, we have 
		\begin{align*}
			r_{k+1} &= \left(1+\sum_{i=0}^{m_k} \beta_i^{(k)} \right) (I-A) r_k - \sum_{i=0}^{m_k}  \beta_i^{(k)} r_{k-i}\\
			&=r_k -\left(1+\sum_{i=0}^{m_k} \beta_i^{(k)} \right) A r_k - \sum_{i=0}^{m_k}  \beta_i^{(k)} (r_k-r_{k-i}). 
		\end{align*}
		Using the previous orthogonality property $r_{k+1}^TAr_k=0$ and $r_{k+1}^T(r_k-r_{k-i})=0$, we have
		\begin{equation}
			r_{k+1}^Tr_{k+1}=r_{k+1}^Tr_k,
		\end{equation}
		that is, $r_{k+1}^T(r_{k+1}-r_k)=0$, which is the desired result. 
	\end{proof}
	
	\subsection{NGMRES($m$) vs. GMRES}\label{subsec:NGMRES1-GMRES}
	In this subsection, we explore the relationship between NGMRES($m$) and GMRES. We start with NGMRES(0).
	
	For $m=0$, let us denote the (single) coefficient, given $k$, as $\beta^{(k)}$. Then, one can show that $x_{k+1}=x_k-(1+\beta^{(k)})r_k$.  By standard calculations, we have
	\begin{equation*}
		r(q(x_k))=Mr_k \quad \text{and} \quad r(q(x_k))-r(x_k)=-Ar_k.
	\end{equation*}
	Thus,
	\begin{equation*}
		\beta^{(k)} =-\frac{r(q(x_k))^T (r(q(x_k))-r(x_k) )}{\|r(q(x_k))-r(x_k) \|^2}=\frac{r_k^TAr_k}{r_k^TA^TAr_k}-1.
	\end{equation*}
	Let $\alpha_k=\frac{r_k^TAr_k}{r_k^TA^TAr_k}$. It follows that $x_{k+1}=x_k-\alpha_k r_k.$
	From Theorem \ref{thm:orthogonality}, we have $r_{k+1}^TAr_k=0$ for all $k$.  Note the first iteration of  GMRES is $x_1 = x_0 -\frac{r_0^T A^T r_0}{r_0^T A^T A r_0} r_0$, which is the same as $x_1$ of NGMRES(0). However, from the next iteration, NGMRES(0) generates different iterates. 
	
	From the above discussion, we can rewrite NGMRES(0) as Algorithm \ref{alg:NGMRES0}, which is exactly the  Minimal Residual (MR) Iteration or GMRES(1), see \cite[Chapter 5.3.2]{saad2003iterative}. 
	\begin{algorithm}[H]
		\caption{NGMRES(0) or Minimal Residual (MR) Iteration or GMRES(1)} \label{alg:NGMRES0}
		\begin{algorithmic}[1] 
			\STATE Given $x_0$, let $r_0=Ax_0-b$
			\FOR {$k=0,1, \cdots$ until convergence}
				\STATE compute $\alpha_k=\frac{r_k^TAr_k}{r_k^TA^TAr_k}$
				\STATE compute $x_{k+1}=x_k-\alpha_k r_k$ 
				\STATE compute $r_{k+1}=r_k-\alpha_kAr_k$
			\ENDFOR
		\end{algorithmic}
	\end{algorithm}

	We now  provide an interesting result for NGMRES($m$), which establishes the relationship among NGMRES($m$), full AA, and GMRES. We denote the iterates and residuals of NGMRES($m$) as $x_j^{NG(m)}$ and $r_j^{NG(m)}$, respectively. In practice, we are interested in $A$ being invertible, so we will focus on this situation.
	\begin{theorem}\label{thm:gmres1=gmres}
		Assume that  NGMRES($m$) with $m\geq 1$, full AA, and classical GMRES use the same initial guess $x_0\neq x^*$,  $A$ is invertible, and $A$ is either symmetric or shifted skew-symmetric of the form $A=\alpha I+S$ where $S$ is skew-symmetric.  Furthermore, assume for some $k_0> 0$,  $r_{k_0-1}^{G}\neq 0$ and $\|r_{k-1}^{G}\|_2 >\|r_k^{G}\|_2$ for each $k$ such that $0 < k < k_0$. Then,  for $0\leq j \leq k_0$,  we have 
		\begin{equation}
			x_j^{NG(m)}=x_j^G, \quad \forall m \in \mathbb{Z}^{+},
		\end{equation}	
		and
		\begin{equation}
			x_{j+1}^A=q(x_j^{NG(m)})=x_j^{NG(m)}-r_j^{NG(m)}, \quad \forall m \in \mathbb{Z}^{+}.
		\end{equation}
		
	\end{theorem}
	Theorem \ref{thm:gmres1=gmres} tells us that under certain conditions, we can generate the iterates obtained from GMRES by NGMRES(1), which is very simple because in each step one only needs to solve a $2\times 2$ linear system that gives the coefficients $ \beta_0^{(k)} $ and $ \beta_1^{(k)} $, to update the iterates $x_{k+1}$. 
	
	Before providing the proof, we first establish some notation for later use.
	Let us write $W_k$ defined in \eqref{eq:definition-Wk} as
	\begin{equation}\label{eq:simply-notation}
		W_k=[w_1, w_2,\cdots, w_{m_k+1}] \in \mathbb{R}^{n\times (m_k+1)}.
	\end{equation}
	Then,
	\begin{equation}\label{eq:Cmk}
		C_{m_k,k}=W_k^TW_k=
		\begin{bmatrix}
			w_1^Tw_1 &  w_1^Tw_2 & \cdots & w_1^Tw_{m_k+1}\\
			w_2^Tw_1 &  w_2^Tw_2 & \cdots & w_2^Tw_{m_k+1}\\
			\vdots   &  \vdots   & \vdots  & \vdots \\
			w_{m_k}^Tw_1 &  w_{m_k}^Tw_2 & \cdots& w_{m_k}^Tw_{m_k+1}\\
			w_{m_k+1}^Tw_1 &  w_{m_k+1}^Tw_2 & \cdots& w_{m_k+1}^Tw_{m_k+1}
		\end{bmatrix}
	\end{equation}
	and 
	\begin{equation}\label{eq:fmk}
		f_{m_k,k}=W_k^TMr_k=
		\begin{bmatrix}
			w_1^TMr_k \\
			w_2^TMr_k\\
			\vdots  \\
			w_{m_k}^TMr_k \\
			w_{m_k+1}^TMr_k
		\end{bmatrix}\in\mathbb{R}^{(m_k+1)}.
	\end{equation}
	 We write $\boldsymbol{\beta}^{(k)}$ in \eqref{eq:beta-form} as $\boldsymbol{\beta}^{(k)}_{m_k}$.  Then, \eqref{eq:beta-form} can be rewritten as
	\begin{equation*}
  C_{m_k,k}\boldsymbol{\beta}^{(k)}_{m_k} =f_{m_k,k}.
	\end{equation*}
	We point out that  $w_j$ in $W_k$  defined in \eqref{eq:simply-notation} and $W_{\ell}$ are different. In other words, 
    $W_k$ is not a submatrix of $W_\ell$ if $\ell>k$.
    
    We first state a lemma that will be used in our proof of Theorem \ref{thm:gmres1=gmres}.
	\begin{lemma}\label{lem:xA^Tyz}
		Assume $A$ is either symmetric or shifted skew-symmetric of the form $A=\alpha I +S$, where $S$ is skew-symmetric. For $x , y ,z\in \mathbb{R}^n$ satisfying $x^TA(y-z)=0$ and $x^T(y-z)=0$, we have
		\begin{equation*}
			x^TA^T(y-z) =0.
		\end{equation*}   
	\end{lemma}
	\begin{proof}
		 If $A$ is symmetric, then $A^T=A$ and it follows trivially that $x^TA^T(y-z) =x^TA(y-z)=0$. If $A$ is shifted skew-symmetric,  from  the conditions $x^TA(y-z)=0$ and $x^T(y-z)=0$, we have $x^TA(y-z)=x^T(\alpha I+S)(y-z)=0+x^T S(y-z)=0$. It follows that $x^TA^T(y-z)=x^T(\alpha I+S)^T(y-z)=0+x^T(-S)(y-z)=0$.  
	\end{proof}

	Given a matrix $B$, let $B(1 : s,1 : s)$ be the top-left $s\times s$ block of $B$. We define the $k$-step iterate in \eqref{eq:xkp1} of NGMRES($m$) as $x_{k+1}^{NG(m)}$. When the iterates for NGMRES and GMRES are the same, we drop the superscript for simplicity. The main idea now is that using Theorem \ref{thm:relation-infinity}, we have ${x}_{k+1}^{NG(\ell)}=x_{k+1}^G$ for any $ \ell\geq k$, and then we only need to prove that $x_{k+1}^{NG(\ell_1)}=x_{k+1}^{NG(k)}$ for $\ell_1=1, 2, \cdots, k-1$.

    We are now ready to present the proof of Theorem \ref{thm:gmres1=gmres}.
	\begin{proof}
		We consider NMGRES($m$) with $m\geq 1$ and the $k$-step update $x_{k+1}$, where $k\leq k_0-1$.
		
		\begin{enumerate}[(i)] 
        \item When $k=0$ and $\forall m\in\mathbb{Z}^+$,  $\min\{k, m\}=0$. From Theorem \ref{thm:relation-infinity} we have $x_1^{NG(m)}=x_1^G$.
		
	\item When $k=1$ and $\forall m\in\mathbb{Z}^+$,  $\min\{k, m\}=1$.  From Theorem \ref{thm:relation-infinity} we have $x_2^{NG(m)}=x_2^G$.
		
	\item  When $k=2$ and $\ell\geq 2$,  $\min\{k, \ell\}=2$.  From Theorem \ref{thm:relation-infinity} we have  $x_3^{NG(\ell)}=x_3^G$ for $ \ell\geq 2$. Now, we only need to show that $x_3^{NG(1)}=x_3^{NG(2)}$. For NGMRES(2),  $m_k=\min\{k, m\}=2$ and
		\begin{equation*}
			W_2=[r_2-Mr_2, r_1-Mr_2, r_0-Mr_2]:=[w_1,w_2,w_3].
		\end{equation*}
        \end{enumerate}
		Next, we show that the corresponding $C_{m_k,k}$ and $f_{m_k,k}$  in \eqref{eq:Cmk} and \eqref{eq:fmk}, respectively, have a special structure:
		\begin{equation}\label{eq:c22}
			C_{2,2}=
			\begin{bmatrix}
				w_1^Tw_1 &  w_1^Tw_2  & w_1^Tw_3\\
				w_2^Tw_1 &  w_2^Tw_2  & w_2^Tw_3\\
				w_{3}^Tw_1 &  w_{3}^Tw_2 & w_{3}^Tw_3
			\end{bmatrix}
			=\begin{bmatrix}
				w_1^Tw_1 &  w_1^Tw_2  & w_1^Tw_2\\
				w_1^Tw_2 &  w_2^Tw_2  &  w_2^Tw_2\\
				w_1^Tw_2 &  w_2^Tw_2 & w_{3}^Tw_3
			\end{bmatrix},
		\end{equation}
		and 
		\begin{equation}\label{eq:f22}
			f_{2,2}=W_2^TMr_2=
			\begin{bmatrix}
				w_1^TMr_2 \\
				w_2^TMr_2\\
				w_2^TMr_2
			\end{bmatrix}.
		\end{equation}
		Equations \eqref{eq:c22} and \eqref{eq:f22} are equivalent to the following three conditions:
		\begin{align*}
			w_1^Tw_2 &= w_1^Tw_3 \quad \Longleftrightarrow (Ar_2)^T(r_1-r_0)=0;\\
			w_2^Tw_2  &= w_2^Tw_3 \quad \Longleftrightarrow (r_1-Mr_2)^T(r_1-r_0)=0;\\	
			w_2^TMr_2  &= w_3^TMr_2 \quad \Longleftrightarrow (Mr_2)^T(r_1-r_0)=0.
		\end{align*}
		Using \eqref{eq:orthogonality-dif} with $k=1, j=0, i=1$, we have $r_2^T(r_1-r_0)=0$  and $r_2^TAr_1=0$. Recall $r_1^T(r_1-r_0)=0$.  We only need the following condition to make the above equalities hold true:
		\begin{equation}\label{eq:r2Ar1r0}
			r_2^TA^T(r_1-r_0)=0.
		\end{equation}
		Using \eqref{eq:GMRES-orth-space} for $k=2$, we have $r_2^TAr_0=0$. It follows that $r_2^TA(r_1-r_0)=0$. From Lemma \ref{lem:xA^Tyz}, Equation \eqref{eq:r2Ar1r0} holds.
		
		For NGMRES(1),  we have	$C_{1,2}= C_{2,2}(1 : 2,1 : 2)$ and $f_{1,2}=f_{2,2}(1 : 2)$. 
		Let $\boldsymbol{\beta}_{1,*}^{(2)}$ be a solution of $C_{1,2}\boldsymbol{\beta}_1^{(2)}=f_{1,2}$. It is obvious that $[\boldsymbol{\beta}_{1,*}^{(2)};0]$ is a solution of $C_{2,2}\boldsymbol{\beta}_2^{(2)}=f_{2,2}$. Thus, $r_3^{NG(1)}=r_3^{NG(2)}$. However, $r_3^{NG(1)}=Ax_3^{NG(1)}-b$, $r_3^{NG(2)}=Ax_3^{NG(2)}-b$ and $A$ is invertible. It follows that $x_3^{NG(1)}=x_3^{NG(2)}$.
		
		We proceed by an induction:
		
		\noindent$\bullet$ We assume that for $k\leq k_*<k_0-1$, it holds that $x_{k+1}^{NG(m)}=x_{k+1}^G$ for $\forall m\in \mathbb{Z}^+$.
		
		\noindent$\bullet$ We now prove that for $k= k_*+1:=p$, it holds   that $x_{k+1}^{NG(m)}=x_{k+1}^G$ $\forall m\in\mathbb{Z}^+$. Using Theorem \ref{thm:relation-infinity}, we have $x_{k+1}^{NG(\ell)}=x_{k+1}^G$ for any $\ell\geq k_*+1$. Next, we show that for $\ell_1=1, 2, \cdots, p-1$ it holds  that  $x_{k+1}^{NG(\ell_1)}=x_{k+1}^{NG(p)}$. Then, we have $x_{k+1}^{NG(m)}=x_{k+1}^G$ for $\forall m\in\mathbb{Z}^+$.
		
		 For NGMRES($p$),  $m_k=\min\{k, p\}=p$ and
		\begin{equation*}
			W_{p,p}=[r_p-Mr_p, r_{p-1}-Mr_p,\cdots, r_0-Mr_p]:=[w_1,w_2,\cdots,w_{p+1}].
		\end{equation*}
		Next, we prove that the corresponding $C_{m_k,k}$ and $f_{m_k,k}$ in \eqref{eq:Cmk} and \eqref{eq:fmk} have the following forms, for $t=2, 3, \cdots, p+1$:
		\begin{equation}\label{eq:property-Cmk-fmk}
			C_{m_k,k}(t,1 : 2)=[w_2^Tw_1, w_2^Tw_2]\quad  \text{and}\quad f_{m_k,k}(t)=w_2^TMr_p,
		\end{equation}
		that is,   
		\begin{align*}
			C_{p,p}&=
			\begin{bmatrix}
				w_1^Tw_1 &  w_1^Tw_2 & \cdots & w_1^Tw_{p+1}\\
				w_2^Tw_1 &  w_2^Tw_2 & \cdots & w_2^Tw_{p+1}\\
				\vdots   &  \vdots   & \vdots  & \vdots \\
				w_{p}^Tw_1 &  w_{p}^Tw_2 & \cdots& w_{p}^Tw_{p+1}\\
				w_{p+1}^Tw_1 &  w_{p+1}^Tw_2 & \cdots& w_{p+1}^Tw_{p+1}
			\end{bmatrix}=\begin{bmatrix}
				w_1^Tw_1 &  w_1^Tw_2 & \cdots & w_1^Tw_2\\
				w_2^Tw_1 &  w_2^Tw_2 & \cdots & w_2^Tw_2\\
				\vdots   &  \vdots   & \vdots  & \vdots \\
				w_2^Tw_1 &  w_2^Tw_2 & \cdots& w_{p}^Tw_{p+1}\\
				w_2^Tw_1 &  w_2^Tw_2 & \cdots& w_{p+1}^Tw_{p+1}
			\end{bmatrix}
		\end{align*}
		and 
		\begin{equation*}
			f_{p,p}=W_{p}^TMr_{p}=
			\begin{bmatrix}
				w_1^TMr_p \\
				w_2^TMr_p\\
				\vdots  \\
				w_p^TMr_p \\
				w_{p+1}^TMr_p
			\end{bmatrix}=	\begin{bmatrix}
				w_1^TMr_p \\
				w_2^TMr_p\\
				\vdots  \\
				w_2^TMr_p \\
				w_2^TMr_p
			\end{bmatrix}.
		\end{equation*}
We are only interested in the first two columns of $C_{p,p}$.	Proving \eqref{eq:property-Cmk-fmk}  is equivalent to proving that for $t=3, 4, \cdots, p+1$ the following three requirements hold:
	\begin{subequations}
    \label{eq:req}
        \begin{align}
			w_1^Tw_2 &= w_1^Tw_t \quad \Longleftrightarrow (Ar_p)^T(r_{p-1}-r_{p+1-t})=0; \label{eq:cond1}\\
			w_2^Tw_2  &= w_2^Tw_t \quad \Longleftrightarrow (r_{p-1}-Mr_p)^T(r_{p-1}-r_{p+1-t})=0; \label{eq:cond2}\\	
			w_2^TMr_2  &= w_t^TMr_2 \quad \Longleftrightarrow (Mr_p)^T(r_{p-1}-r_{p+1-t})=0.\label{eq:cond3}
		\end{align}
        \end{subequations}
		Below, we prove the requirements of \eqref{eq:req}. 
		Since $k=k_*+1=p$, we have $k-1=k_*=p-1$. For the $k_*$-step NGMRES($k_*$), using \eqref{eq:orthogonality-dif}, i.e., $r_{k_*+1}^TAr_{k_*}=0$ and $r_{k_*+1}^T(r_{k_*-j}-r_{k_*-i})=0$ with $ j=0, i=t-2$, we have
		\begin{equation}\label{eq:orth-rpArp}
			r_p^TAr_{p-1}=0 \quad \text{and} \quad r_{p}^T(r_{p-1}-r_{p+1-t})=0.	
		\end{equation} 
		Recall the assumption that for $j\leq k_*$,  NGMRES($m$) is the same as GMRES for 
        $\forall m\in \mathbb{Z}^+$. Since $r_{p+1-t}\in \mathcal{K}_{p+1-t+1}$ and $p+2-t\leq p-1$,  using  \eqref{eq:GMRES-orth-space} gives
		\begin{equation}\label{eq:orth-rpArpt}
			r_p^T A r_{p+1-t}=0.
		\end{equation}
		Moreover, using the second equality in \eqref{eq:GMRES-pro}, i.e., $(r_{k+1}^G)^T(r_j^G-r_i^G) =0$ with $k=p-2, j=p-1, i=p+1-t$, we obtain 
		\begin{equation}\label{eq:orth-rpm1Arpt}
			r_{p-1}^T(r_{p-1}-r_{p+1-t})=0.	
		\end{equation}
		Using \eqref{eq:orth-rpArp} and \eqref{eq:orth-rpArpt}, we have 
		\begin{equation}\label{eq:orth-rpArpmpt}
			r_p^TA(r_{p-1}-r_{p+1-t})=0	.
		\end{equation}
		Using  \eqref{eq:orth-rpm1Arpt} and \eqref{eq:orth-rpArpmpt},  Lemma \ref{lem:xA^Tyz} gives us
		\begin{equation}\label{eq:orth-rpATrpmpt}
			r_p^TA^T(r_{p-1}-r_{p+1-t})=0.
		\end{equation}
		Thus, using the second equality in \eqref{eq:orth-rpArp}, \eqref{eq:orth-rpm1Arpt} and \eqref{eq:orth-rpATrpmpt} gives  \eqref{eq:cond1}, \eqref{eq:cond2}, and \eqref{eq:cond3}.
		
		Next, we show $r_{k+1}^{NG(\ell_1)}=r_{k+1}^{NG(p)}$, where $\ell_1= 1, 2, \cdots, p-1$. For NGMRES($\ell_1$), the coefficient matrix of the normal equations \eqref{eq:normal-equ} is $C_{\ell_1,k}=C_{p,p}(1:\ell_1+1,1:\ell_1+1)$, and $f_{\ell_1,k}=f_{p,p}(1:\ell_1+1)$. For NGMRES(1), let $\boldsymbol{\beta}_{1,*}^{(k)}$ be a solution of $C_{1,k}\boldsymbol{\beta}_1^{(k)}=f_{1,k}$. Then, $\boldsymbol{\beta}_{\ell_1,*}^{(k)}=[\boldsymbol{\beta}_{1,*}^{(k)};0_{\ell_1-1}]$, where $0_{\ell_1-1}$ is a column zero vector with length $\ell_1-1$,  is a solution of 
		$C_{\ell_1,k}\boldsymbol{\beta}_{\ell_1}^{(k)}=f_{\ell_1,k}$, i.e., a solution of the least-squares problem, \eqref{eq:LSP-linear}. It follows that $r_{k+1}^{NG(\ell_1)}=r_{k+1}^{NG(p)}$. Since $A$ is invertible, we have $x_{k+1}^{NG(\ell_1)}=x_{k+1}^{NG(p)}=x_{k+1}^{G}$ for $\ell_1= 1, 2, \cdots, p-1$.  Since $x_{k+1}^{NG(\ell)}=x_{k+1}^G $ for any $\ell\geq k_*+1=p$, we have $x_{k+1}^{NG(m)}=x_{k+1}^G$ for $m\in\mathbb{Z}^+$.
	\end{proof}

 	
	\subsection{Three-term recurrence for NGMRES(1)}
	For $m=1$ and given $x_0$, NGMRES(1) gives
	\begin{align}
		x_1 &= x_0 -\frac{r_0^T A^T r_0}{r_0^T A^T A r_0} r_0,\\ 
		x_{k+1} &= q(x_k) + \beta_0^{(k)} \left(q(x_k)-x_{k} \right) + \beta_1^{(k)} \left(q(x_k)-x_{k-1} \right), \quad k\geq 1 \label{eq:xk1-gmres1},
	\end{align}
	where the coefficients $ \beta_0^{(k)} $ and $ \beta_1^{(k)} $ are given by \eqref{eq:beta-form}. 
	We thus have
	\begin{equation}
		W_{1,k}=[r_k-Mr_k, r_{k-1}-Mr_k]=[Ar_k, Ar_k+r_{k-1}-r_k]:=[w_1,w_2],
        \label{eq:W1k}
	\end{equation}
	and the matrix and right-hand side of the least-squares problem are given, respectively,  by
	\begin{equation}\label{eq:ngmres1WkWk}
		C_{1,k}=
		\begin{bmatrix}
			w_1^Tw_1 &  w_1^Tw_2 \\
			w_2^Tw_1 &  w_2^Tw_2
		\end{bmatrix}
	\end{equation}
	and 
	\begin{equation}\label{eq:ngmres1f1k}
		f_{1,k}=W_2^TMr_k=
		\begin{bmatrix}
			w_1^TMr_k \\
			w_2^TMr_k
		\end{bmatrix}.
	\end{equation}
	We rewrite \eqref{eq:xk1-gmres1} as
	\begin{equation}
		x_{k+1}=x_k-(1+\beta_0^{(k)}+\beta_1^{(k)})r_k + \beta_1^{(k)}(x_k-x_{k-1}).
        \label{eq:ngmres1xkp1}
	\end{equation}
	The iterative procedure for NGMRES(1) is thus as given in Algorithm \ref{alg:NGMRES1}. The linear system solve can be replaced by a step of Gram-Schmidt, but since the normal equations are a $2 \times 2$ linear system, there is no practical difference between the two approaches.
    
    When $A$ is symmetric, $w_1^TMr_2=-w_1^Tw_2$ because $(Ar_k)^Tr_{k-1}=0$. In this case, NGMRES(1) is GMRES, which in turn can be rewritten as the Conjugate Residual Method (CR)  \cite[Chapter 6.8]{saad2003iterative}. Thus, NGMRES(1) is mathematically equivalent to CR. However, CR updates $x_k$ in a different way. Moreover, in Algorithm \ref{alg:NGMRES1} if we replace $r_k$ by $Ax_k-b$, then the update $x_{k+1}$ is a  three-term recurrence. This three-term recurrence relation does not mean that the corresponding polynomials of the residuals or iterates are orthogonal polynomials, because the NGMRES(1) algorithm does not rely on an orthogonalization mechanism similar to the Arnoldi process. However, the equalities \eqref{eq:orthogonality-dif} and \eqref{eq:NGMRESm-decrease} in Theorem \ref{thm:orthogonality} reflect some ``local'' orthogonality properties of NGMRES($m$) with $m\geq 1$ and general $A$. 
    For example, $r_{k+1}^TAr_k=0$ indicates that the residual polynomials $p_{k}$ (see Theorem \ref{thm:NGMRESm-poly} and recall that $M=I-A$) satisfy $ \langle p_{k},p_{k+1} \rangle_A=(p_{k}(I-A)v, p_{k+1}(I-A)v)_A=0$, where we define $(u,v)_A=v^TAu$. These results, of course, hold for $m=1$ and symmetric $A$.
	\begin{algorithm}[H]
		\caption{NGMRES(1): three-term recurrence }\label{alg:NGMRES1}
		\begin{algorithmic}[1] 
        \REQUIRE Symmetric matrix $A$, right-hand side $b$, initial guess $x_0$
			\STATE $r_0=Ax_0-b$
			\STATE  $x_1=x_0-\frac{r_0^TAr_0}{r_0^TA^TAr_0} r_0$
			\FOR {$k=1,2,\dots$ until convergence}
				\STATE  $r_k=Ax_k-b$
				\STATE Solve $C_{1,k}\boldsymbol{\beta}_1^{(k)}=f_{1,k}$, where $C_{1,k}$ and $f_{1,k}$ are given in \eqref{eq:ngmres1WkWk} and \eqref{eq:ngmres1f1k}, and $\boldsymbol{\beta}_1^{(k)}=[\beta_0^{(k)}; \beta_1^{(k)}]$
				\STATE  $x_{k+1}=x_k-(1+\beta_0^{(k)}+\beta_1^{(k)})r_k + \beta_1^{(k)}(x_k-x_{k-1})$
			\ENDFOR
		\end{algorithmic}
	\end{algorithm}

\begin{algorithm}
\label{alg:CR}
\caption{Conjugate Residual (CR) method \cite[Algorithm 6.20]{saad2003iterative}}
\begin{algorithmic}[1]
\REQUIRE Symmetric matrix $A$, right-hand side $b$, initial guess $x_0$
\STATE $r_0 = b - A x_0$
\STATE $p_0 = r_0$
\FOR{$k = 0, 1, 2, \dots$ until convergence}
    \STATE $\alpha_k = \dfrac{ r_k^T A r_k}{(A p_k)^T (A p_k)}$
    \STATE $x_{k+1} = x_k + \alpha_k p_k$
    \STATE $r_{k+1} = r_k - \alpha_k A p_k$
   
    \STATE $\beta_k = \dfrac{r_{k+1}^T A r_{k+1}}{r_k^T A r_k}$
    \STATE $p_{k+1} = r_{k+1} + \beta_k p_k$
    \STATE Compute $A p_{k+1} =  A r_{k+1} + \beta_k A p_k$
\ENDFOR
\end{algorithmic}
\end{algorithm}

For completeness, we include CR in Algorithm \ref{alg:CR}, following \cite{saad2003iterative},  and further comment on the computational complexity of NGMRES(1)   vs. GMRES in the symmetric case (which is mathematically equivalent to CR).
Note that as a consequence of Theorem \ref{thm:gmres1=gmres}, when $A$ is symmetric,  we can generate the iterates obtained from GMRES  by NGMRES(1). 

Examining Algorithm \ref{alg:NGMRES1} and Equations \eqref{eq:W1k}--\eqref{eq:ngmres1xkp1}, 
NGMRES(1) requires one matrix-vector product, $A r_k$. 
 Four inner products are required to generate the normal equations: $(A r_k)^T (A r_k)$,  $ r_k^T A r_k$, 
 $r_k^T r_k$, and $r_{k-1}^T r_k$. (Notice that by Theorem \ref{thm:orthogonality}, the inner product $r_{k-1}^T A r_k$, which appears in the mathematical formula for forming the normal equations, vanishes and need not be computed.)
 
 Computing $x_{k+1}$ using \eqref{eq:ngmres1xkp1} requires two vector updates (axpy).
Computing the solution of the $2 \times 2$ linear system defined in \eqref{eq:ngmres1WkWk}--\eqref{eq:ngmres1f1k}, once it has been formed, is trivial and requires a negligible number of floating-point operations. In terms of storage, other than the matrix, four vectors are needed: $x_k, x_{k-1}, r_k, Ar_k$.

An efficient implementation of CR 
in Algorithm \ref{alg:CR} requires one matrix-vector product; in line 9 of the algorithm, $A p_{k+1}$ can be obtained from $A r_{j+1}$. Two inner products are needed: $(Ap_k)^T (Ap_k)$ and $r_{k+1}^T A r_{k+1}$. (Other quantities, such as $ r_k^T A r_k$, are re-used.) Three vector updates (axpy) are needed, for $x$, $r$, and $p$. In terms of storage, other than the matrix, five vectors are needed for the $k$th iterate: $x, p, r, Ar, Ap$.

We see then that the overall computational cost of the two methods and storage requirements are rather similar but not identical. Both require one matrix-vector product; NGMRES(1) requires two more inner products compared to CR, but one fewer vector update and one fewer vector to store.

Given that the two methods have different implementations, it may be useful to analyze their numerical stability properties in a finite precision environment. This remains an item for future work.

	\begin{remark}
	In Theorem \ref{thm:gmres1=gmres}, when $A=\alpha I+ S$, where $\alpha=0$ and $S$ is skew-symmetric,  no  $k_0>0$  exists, because in Algorithm \ref{alg:NGMRES1}, $x_1=x_0$ for both GMRES and NGMRES.  In this situation, NGMRES($m$) will make no progress, i.e., $x_k=x_0$ for $k>0$.  We note, however, that special short-recurrence versions of GMRES exist for purely skew-symmetric linear systems; see, for example, \cite{greif2016numerical, greif2009iterative} and the references therein.
	\end{remark}

	\section{Convergence analysis}\label{sec:conv}
	In this section, we conduct convergence analysis for NGMRES. From Theorem \ref{thm:orthogonality}, we know that the 2-norm of the residuals generated by NGMRES are  nonincreasing. Thus, the sequence $\{\|r_k\|\}$ always converges. If $\{\|r_k\|\}$  converges to zero, then  $\{x_k\}$ converges to the exact solution $x^*$. However, if $\{\|r_k\|\}$ converges to a nonzero vector, then $\{x_k\}$ does not converge to $x^*$. In the following, we will present a concrete convergence analysis for NGMRES.

	We call $A$ a real positive definite matrix (or positive real)  if $A+A^T$ is symmetric positive definite. From \eqref{GMRES-LSP-rk}, we know that NGMRES($m$) minimizes the current residuals $r_{k+1}$, and NGMRES(0) is the same as MR.  Thus, the result of MR offers an upper bound on the convergence rate of NGMRES($m$). 
	
	\begin{theorem}\label{thm:NGMRES-converge-pd} 
		Let $A$ be a real positive definite matrix, and $\mu=\lambda_{\rm min}(A+A^T)/2$, $\nu=\lambda_{\rm min}(A^{-1}+(A^{-1})^T)/2$ and $\sigma=\|A\|_2$. Then, NGMRES($m$) converges for any initial guess $x_0$ and the corresponding residuals satisfy the relation
		\begin{equation}\label{eq:NGMRESm-bd1}
			\|r_{k+1}\|_2\leq \sqrt{1-\frac{\mu^2}{\sigma^2}}\,\|r_k\|_2,
		\end{equation}
		\begin{equation}\label{eq:NGMRESm-bd2}
			\|r_{k+1}\|_2\leq \sqrt{1- \mu \nu}\,\|r_k\|_2,
		\end{equation}
	\end{theorem}
	\begin{proof}
		When $m=0$, NGMRES(0) is MR, and  the above results can be directly inferred from \cite[Theorem 5.10]{saad2003iterative}.
		When $m>0$,  using 
		\begin{align*}
			\|r_{k+1}\|_2^2&=\min_{\boldsymbol{\beta}^{(k)}}\|\left(1+\sum_{i=0}^{m_k} \beta_i^{(k)} \right)  M r_k
			-\sum_{i=0}^{m_k} \beta_i^{(k)}   r_{k-i}\|^2_2 \\
			&\leq\min_{\beta_0^{(k)}}\|(1+\beta_0^{(k)})Mr_k-\beta_0^{(k)}r_k\|^2_2 \\
			&=\|(I-\alpha_k A)r_k\|^2_2,
		\end{align*}
		we obtain the desired result.
	\end{proof}

	For $m=0$,  AA($m$) is just the fixed-point iteration, $q(x_k)=Mx_k+b$. Then, AA($m$) converges if and only if $\|M\|<1$. In contrast, Theorem \ref{thm:NGMRES-converge-pd} tells us that for NGMRES(0) we only need to require that $A$ is a real positive definite matrix to guarantee convergence. In this situation, NGMRES(0) indeed accelerates the original fixed-point iteration, which is AA(0).
	
	We comment that the upper bounds \eqref{eq:NGMRESm-bd1} and \eqref{eq:NGMRESm-bd2} might not be sharp. Further investigation is needed; this is beyond the scope of this work.
	
	When $M=I-A$ is either symmetric or skew-symmetric, NGMRES(0) (which is identical to GMRES(1)) convergence has been studied in \cite{he2025gmres1}. Again, the result can offer an upper bound on the convergence of NGMRES($m$). We state it below.
	\begin{theorem}\label{thm:NGMRESm-sym-skew}
		When $A$ is a symmetric (positive or negative) definite matrix, NGMRES($m$) converges for any initial guess $x_0$ and the corresponding residuals satisfy the relation
		\begin{equation}\label{eq:NGMRES0-defi}
			\|r_{k+1}\|_2\leq \frac{|\lambda_{\max}(A)-\lambda_{\min}(A)|}{|\lambda_{\max}(A)+\lambda_{\min}(A)|}\|r_k\|_2.
		\end{equation} 
		When $M=I-A$ is skew-symmetric, then NMGRES($m$) converges for any initial guess $x_0$ and the corresponding residuals satisfy the relation
		\begin{equation}\label{eq:NGMRES0-skew}
			\|r_{k+1}\|_2\leq \frac{\rho(M)}{\sqrt{1+\rho(M)}}\|r_k\|_2,
		\end{equation} 
		where $\rho(M)$ is the spectral radius of $M$. 
	\end{theorem}
	\begin{proof}
		Using the fact that
		\begin{align*}
			\|r_{k+1}\|_2&=\min_{\boldsymbol{\beta}^{(k)}}\|\left(1+\sum_{i=0}^{m_k} \beta_i^{(k)} \right)  M r_k
			-\sum_{i=0}^{m_k} \beta_i^{(k)}   r_{k-i}\|_2 \\
			&\leq\min_{\beta_0^{(k)}}\|(1+\beta_0^{(k)})Mr_k-\beta_0^{(k)}r_k\|_2 \\
			&=\|(I-\alpha_k A)r_k\|_2,
		\end{align*}
		and the result for NGMRES(0),  see \cite{he2025gmres1}, leads to the desired result.
      
	\end{proof}
	
	We comment that  bounds \eqref{eq:NGMRES0-defi} and \eqref{eq:NGMRES0-skew} are the worst-case root-convergence factor of NGMRES(0), i.e., GMRES(1), and can be attained; see \cite{he2025gmres1}. The convergence rate of GMRES(1) depends on the initial guess.  In \cite{he2025gmres1}, we have  shown that if $A$ is indefinite, the performance of GMRES(1) is highly dependent on the initial guess, and an initial guesses for which it may diverge can be constructed.  
	
	We comment that the result in Theorem \ref{thm:NGMRESm-sym-skew} might not be sharp. However, it indicates that NGMRES($m$) ($m>0$) will not be worse than NGMRES(0).
	
	When $A$ is real symmetric positive definite and we assume that the residual norm of GMRES is strictly decreasing, then according to    Theorems \ref{thm:relation-infinity} and \ref{thm:gmres1=gmres}, GMRES, full NGMRES, and  NGMRES($m$) generate  the same iterates. In this situation, NGMRES(1) is preferred over NGMRES($m$) with $m>1$.  As a result, the convergence analysis for GMRES can be applied to NGMRES(1). 
	\begin{corollary}\label{corollary:NGMRESm-r-converge-PSD} 
		Let $A$ be a real symmetric definite matrix. Then, NGMRES(1) converges for any initial guess $x_0$ and the corresponding residuals satisfy the relation
		\begin{equation}
			\|r_{k+1}\|_2\leq 2\left( \frac{\sqrt{\kappa}-1}{\sqrt{\kappa}+1}\right)^{k+1}\,\|r_0\|_2,
		\end{equation}
		where $\kappa=\frac{\lambda_{\rm max}(A)}{\lambda_{\rm min}(A)}$ if $A$ is symmetric positive definite matrix and $\kappa=\left|\frac{\lambda_{\rm min}(A)}{\lambda_{\rm max}(A)}\right|$ if $A$ is symmetric negative definite matrix. 
	\end{corollary}
	\begin{proof}
		This can be directly derived using Chebyshev polynomials (see, e.g.,  \cite[Chapter 3]{greenbaum1997iterative}).  
	\end{proof}
	
	\section{Numerical experiments}\label{sec:num}
	 
	In this section, we present some numerical results that validate our theoretical findings. We keep our numerical illustrations brief.
  
From \eqref{eq:definition-Wk} it can be observed that when $W_k$ is updated in the next iteration, an inexpensive procedure is required:
		$W_k=  D_k- E_k,
	$ where 
    $$D_k = \begin{bmatrix}
			r_k&  r_{k-1} &\cdots & r_{k-m_k}
		\end{bmatrix}$$
    and $E_k$ is a rank-one matrix whose columns are all identical to $M r_k$: 
    $$ E_k = F \otimes M r_k,$$ where $F$ is a column vector of all ones and $\otimes$ is the Kronecker product.  When $k \to k+1$, to obtain $W_{k+1}=D_{k+1}-E_{k+1}$, all columns but the first column of $D_k$  are just shifted, and the last column is discarded: $D_{k+1}$ will contain a new first column, $r_{k+1}$, and its next $m_k$ columns will be equal to the first through $(m_k-1)$st column of $D_k$, which are nothing but the last residuals. The matrix $E_{k+1}$  is trivial to construct, being rank-one. Therefore, the update of $W_k$ in every iteration amounts to an insertion of one column as the first new column, the deletion of the last column, and a rank-one correction with respect to the previous iteration. Fast techniques such as the ones discussed in \cite{daniel1976reorthogonalization} for the QR factorization may then be applied. 

\begin{example} \label{ex:conv-diff}
We consider the finite difference discretization of a convection-diffusion  model equation in two dimensions on the unit square, with constant convective coefficients:
 	\begin{equation}
 		-\Delta u+ \sigma   u_x+ \tau u_y=f.
 	\end{equation}
 Here, $\sigma$ and $\tau$ are the convective coefficients, and $f=f(x,y)$ is a known function. We discretize the equation on a uniform mesh whose size is $h$, and assume homogeneous Dirichlet boundary conditions.  We use a centered second-order five-point finite difference scheme for the Laplacian, and centered second-order schemes for the first partial derivatives. The discretization results in a system of $n^2 \times n^2$ linear equations, where $n$ is directly connected to the mesh size: $h = 1 / (n+1).$ Denoting the matrix of the linear system by $K$, we observe that for any nonzero $\sigma$ or $\tau$, $K$ is nonsymmetric. We denote the mesh Reynolds numbers by $\gamma_1 = \frac{\sigma h}{2}$ and
 $\gamma_2 = \frac{\tau h}{2}$. 
We consider $\gamma_1=\gamma_2=0.5$ and experiment with a linear system $Ax=b$ for two scenarios: (i) we take $A$ to be $K$; (ii) we denote $M=\frac{K-K^T}{2}$ and take $A=I-M$. The right-hand side vectors are set so that our solution is a vector of all 1s, and we start with a random initial guess. We note that there is an insignificant difference in the iteration counts between different choices of the initial guess or the right-hand side. 

Figure~\ref{fig:conv_diff} shows results for the two experiments with $n=32$ (matrix dimensions $1,024 \times 1,024$).
The left-hand side figure is consistent with our theoretical observation that NGMRES(1) and GMRES are different than one another for nonsymmetric linear systems.
The right-hand side figure confirms our theoretical analysis for a shifted skew-symmetric system (namely, that NGMRES(1) and GMRES are identical). 
\begin{figure}[H]
		\centering
		\includegraphics[width=6cm]{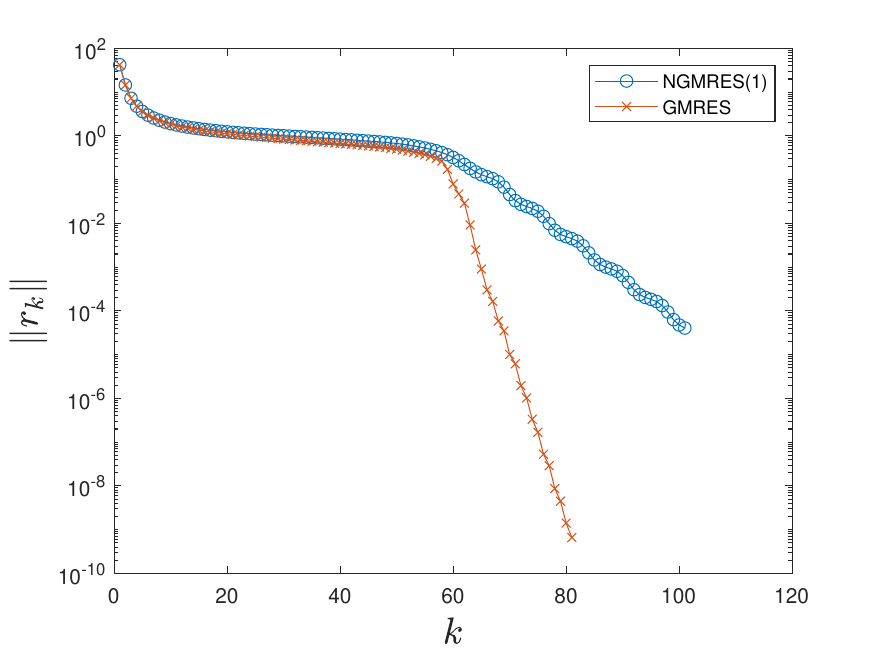}
		\includegraphics[width=6cm]{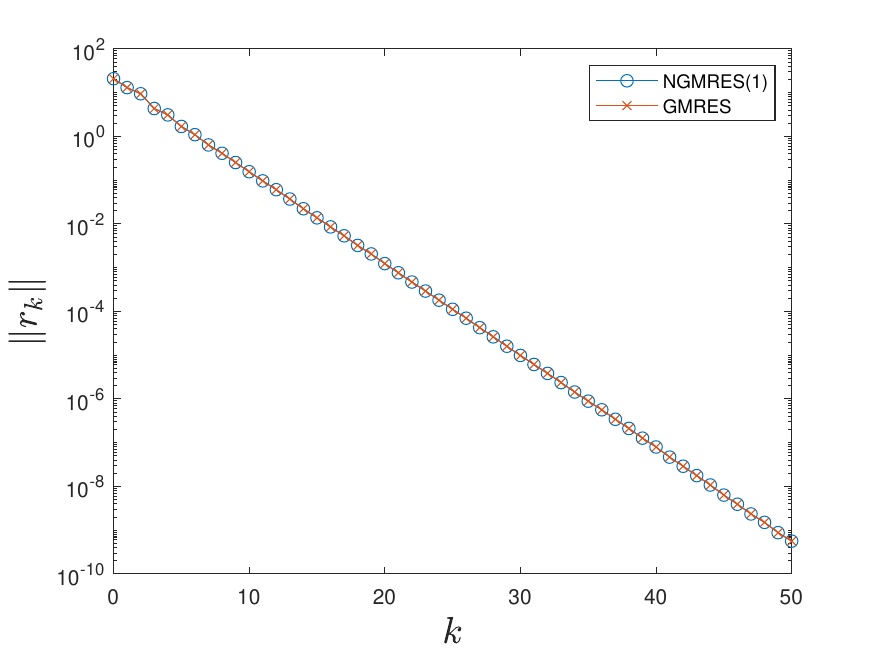}
		\caption{Example \ref{ex:conv-diff}. Convergence history (residual norms) for NGMRES(1) vs GRMES. The left-hand side graph is for a nonsymmetric system. The right-hand side graph is for a shifted skew-symmetric system. }\label{fig:conv_diff}
\end{figure}
\end{example}
 For  shifted-skew symmetric cases using different values of $n$, $\sigma$, and $\tau$, the residuals of the GMRES algorithm and those of NGMRES(1) or full NGMRES may not coincide once the residual norm becomes sufficiently small. For example, consider $n = 10$, $\sigma = \tau = 0.5$, and $A = I - (K - K^T)$. We have observed that after the residual norm drops below $10^{-8}$, GMRES and NGMRES(1) or full NGMRES exhibit a slightly different behavior, possibly due to the least-squares problem featuring a rank-deficient matrix.

\begin{example}\label{ex:circ-staga}
We consider  $Ax=b$ taken from \cite[Chapter 6, Problem P-6.8]{saad2003iterative},  where 
\begin{equation}
   A= \begin{bmatrix}
     0 & 0 & 0 &0 & 1\\
     1&0 & 0 &0 &0 \\
     0& 1& 0& 0 &0 \\
    0& 0& 1& 0 &0\\ 
     0& 0& 0& 1 &0
    \end{bmatrix}, \quad 
    b=\begin{bmatrix}
     1\\0\\0\\0\\0
    \end{bmatrix},
\end{equation}
with inital guess $x_0=[0, 0, 0, 0,0]^T$.
\end{example}
The matrix $A$ is invertible and the exact solution is $x^*=[0, 0, 0, 0, 1]^T$. For this example, we derive the theoretical iterates  rather than run our numerical code. It can be shown that the iterates of GMRES are
\begin{equation}
x_4^G=x_3^G=x_2^G=x_1^G=x_0, \quad x_5^G=x^*.
\end{equation}
Thus, no $k_0$ exists in Theorem \ref{thm:relation-infinity},  and the theorem  does not tell us anything about the relationship between GMRES and NGMRES.

For full NGMRES, $ x_1^{NG}=x_0$. It follows that $W_1$ is rank deficient. We assume that full NGMRES returns the minimum-norm solution  of the least-squares problem \eqref{eq:LSPrksame}. It can be theoretically shown that
\begin{equation}
    x_k^{NG}=x_0, \,\,  \forall k>0.
\end{equation}

When NGMRES(0) is applied to Example \ref{ex:circ-staga}, it  does not make any progress either, i.e., $x_k^{NG}=x_0$ for $k>0$. 

If we consider a different initial guess, $x_0=[1, 1, 1, 1 ,1]^T$, then full NGMRES and GMRES generate the same iterates and converge to the exact solution $x_5=x^*$. We also note that $W_k$ is full rank before full NGMRES  converges to $x^*$.

  In the next example, we compare the performance of NGMRES($m$) and GMRES for nonsymmetric $A$.
\begin{example}\label{ex:cir-n50}
We extend Example \ref{ex:circ-staga} to a matrix $A$ of size $n=50$. We set $x_0$  to be a vector of all 1s and run NGMRES(10) with 50 iterations. Figure \ref{fig:A50} shows the performance of NGMRES(10)  compared to GMRES.  The figure on the left-hand side displays the residual norms of the first 50 iterations. It indicates that when $A$ is nonsymmetric, NGMRES($10$) performs differently than GMRES, with GMRES having slightly lower residual norms throughout the first 49 iterations and converging to the exact solution at the 50th iteration. The figure on the right-hand side zooms on the first 49 iterations. It confirms that the first 11 iterations are identical to machine precision (using double precision), but later iterations differ between the two methods. Furthermore, NGMRES(10) seems to stall early in the iterations. This example suggests that NGMRES($m$) does not have a finite convergence property and that effective preconditioners for NGMRES are needed for hard problems.
\end{example}
	\begin{figure}[H]
		\centering
        \includegraphics[width=6cm]{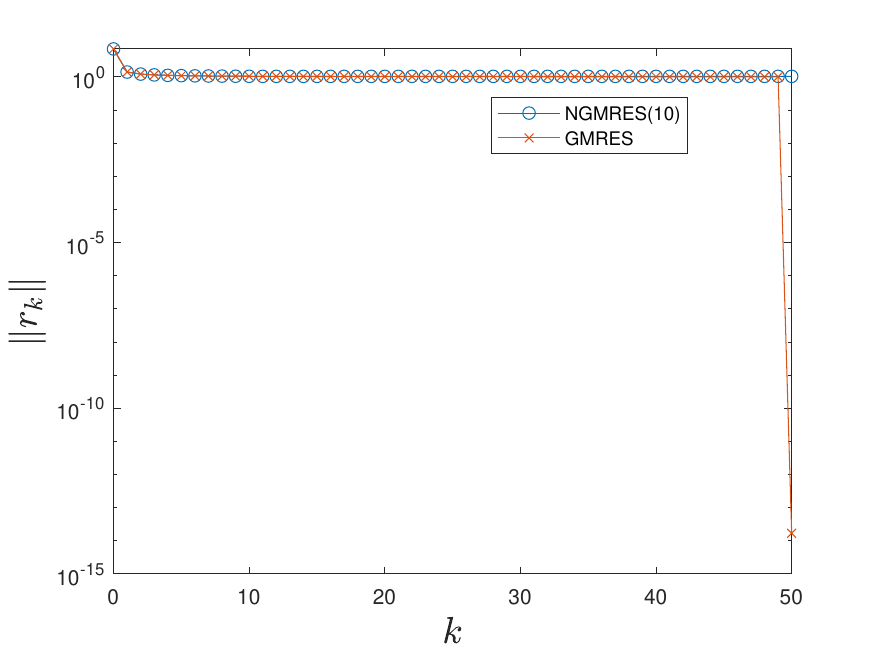}
		\includegraphics[width=6cm]{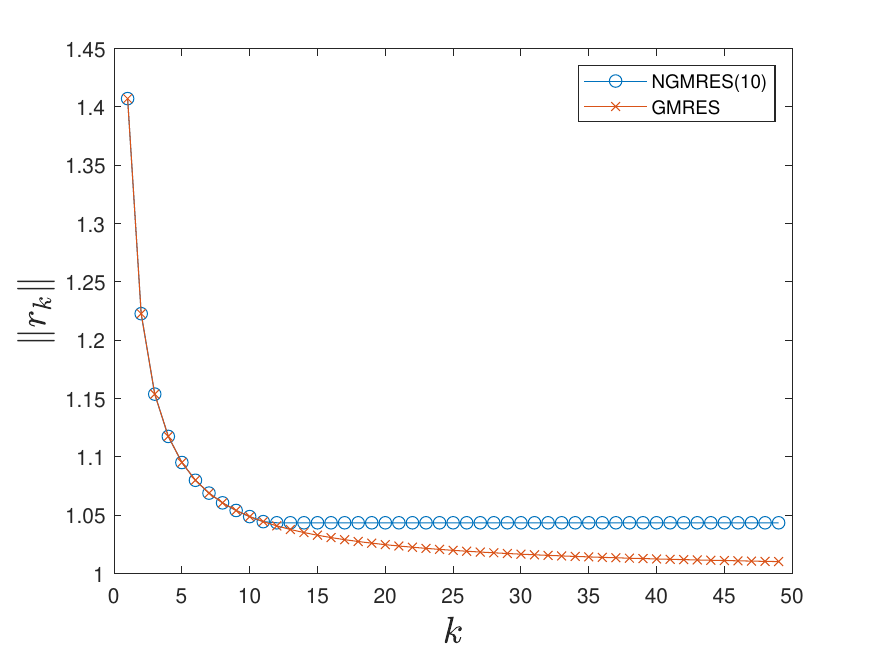}
		\caption{Example \ref{ex:cir-n50}. Convergence history (residual norms) for NGMRES(10) vs. GRMES. The
left-hand side graph displays the first 50 iterations. The right-hand side graph displays the first 49 iterations (droped the initial guess).}\label{fig:A50}
	\end{figure}

	\section{Concluding remarks}\label{sec:con}
	
	We have studied the properties of NGMRES applied to linear systems, which we summarized in the introduction and proved throughout the paper. 	Several questions remain open. For example, can the convergence behavior of NGMRES be characterized more precisely when the coefficient matrix is not real positive definite, and what can be said about convergence for specific scenarios of the spectral distribution of the matrix? 
    
    The development of new efficient variants of NGMRES in combination with preconditioning is critical in practical applications and may provide a rich basis for further exploration, both theoretically and in terms of a practical efficient implementation. 
    
    The implementation or development of efficient and stable solvers for the least-squares problems involved in NGMRES is an interesting topic, especially when the matrix of the least-squares problem is rank-deficient.  
    
    Finally, investigating stagnation situations is a challenging and important issue. 
 
\section*{Acknowledgments}
The authors are grateful to the referees for their very helpful comments and suggestions.

\bibliographystyle{siam}
\bibliography{NGMRESbib}

\end{document}